\documentclass[oneside, 12pt]{amsart}

\usepackage[utf8]{inputenc}
\usepackage[margin=1in]{geometry}
\usepackage{amsthm}
\usepackage{enumitem}  
\usepackage[OT2,T1]{fontenc}
\usepackage{amscd, amssymb, enumitem, amsmath,mathrsfs, amsfonts, amsaddr}
\usepackage{url}
\usepackage{tikz}
\usepackage{fullpage}
\usepackage{microtype}
\usetikzlibrary{decorations.pathreplacing}
\usepackage[backref=page]{hyperref}
\usepackage{hyperref}
\usepackage[utf8]{inputenc}
\usepackage[main=english,russian]{babel}

\DeclareSymbolFont{cyrletters}{OT2}{wncyr}{m}{n}
\DeclareMathSymbol{\Sha}{\mathalpha}{cyrletters}{"58}

\usepackage[utf8]{inputenc}

\newtheorem{theorem}{Theorem}[section]
\newtheorem{lemma}[theorem]{Lemma}
\newtheorem{proposition}[theorem]{Proposition}
\newtheorem*{proposition*}{Proposition}

\newtheorem*{questiona*}{Question A}
\newtheorem*{questionb*}{Question B}
\newtheorem*{theorem*}{Theorem}

\theoremstyle{definition}

\newtheorem{example}[theorem]{Example}
\newtheorem{conjecture}[theorem]{Conjecture}
\newtheorem{claim}[theorem]{Claim}
\newtheorem{remark}[theorem]{Remark}
\newtheorem{emptyremark}[theorem]{}
\newtheorem*{acknowledgement}{Acknowledgements}

\theoremstyle{remark}

\title{A divisibility related to the Birch and Swinnerton-Dyer conjecture}
\author{Mentzelos Melistas}
\address{ ~~~~~ email: mentzmel@gmail.com}
\date{\today}

\usepackage{draftwatermark}
\SetWatermarkText{}
\SetWatermarkScale{3}

\begin{document}

\maketitle

\begin{abstract}
Let $E/\mathbb{Q}$ be an optimal elliptic curve of analytic rank zero. It follows from the Birch and Swinnerton-Dyer conjecture for elliptic curves of analytic rank zero that the order of the torsion subgroup of $E/\mathbb{Q}$ divides the product of the order of the Shafarevich--Tate group of $E/\mathbb{Q}$, the (global) Tamagawa number of $E/\mathbb{Q}$, and the Tamagawa number of $E/\mathbb{Q}$ at infinity. This consequence of the Birch and Swinnerton-Dyer conjecture was noticed by Agashe and Stein in 2005. In this paper, we prove this divisibility statement unconditionally in many cases, including the case where the curve $E/\mathbb{Q}$ is semi-stable.
\end{abstract}


\section{Introduction}

Let $A/\mathbb{Q}$ be an abelian variety with analytic rank $0$ which is a quotient of the modular Jacobian $J_0(N)/\mathbb{Q}$ attached to a newform. Assume that $A/\mathbb{Q}$ is optimal, i.e., the associated map $J_0(N) \rightarrow A$ has connected kernel. It follows from work of Kolyvagin and Logach\"{e}v  (see \cite{kolyvaginlogachev89} and \cite{kolyvaginlogachev92}) that both $A(\mathbb{Q})$ and the Shafarevich--Tate group $\Sha(A/\mathbb{Q})$ of $A/\mathbb{Q}$ are finite. Let $L(A,s)$ be the $L$-function of $A/\mathbb{Q}$ and let $m_A$ be the Manin constant of $A/\mathbb{Q}$. Agashe and Stein in \cite[Page 468]{agashestein} proved that
\begin{align*}
    c_{\infty}(A) \cdot m_A \cdot |A(\mathbb{Q})_{\rm{tors}}| \cdot \frac{L(A,1)}{\Omega(A)} \in \mathbb{Z},
\end{align*}
where $c_{\infty}(A)$ is the number of connected components of $A(\mathbb{R})$ and $\Omega(A)$ is the real period of $A/\mathbb{Q}$. 

The (second part of the) Birch and Swinnerton-Dyer conjecture for analytic rank $0$ abelian varieties (see \cite[Conjecture F.4.1.6]{hindrysilverman}) is the following.
\begin{conjecture}\label{bsdrank0}
    Let $A/\mathbb{Q}$ be an abelian variety with $L$-function $L(A,s)$. If $L(A,1) \neq 0$, then $$\frac{L(A,1)}{\Omega(A)}=\frac{|\Sha(A/\mathbb{Q})|\cdot \prod_{p} c_p(A)}{|A(\mathbb{Q})_{\textrm{tors}}| \cdot |A^{\vee}(\mathbb{Q})_{\textrm{tors}}|},$$
 where $c_p(A)$ denotes the Tamagawa number of $A/\mathbb{Q}$ at $p$ and $A^{\vee}/\mathbb{Q}$ is the abelian variety dual to $A/\mathbb{Q}$. 
\end{conjecture}

Rearranging the equation of Conjecture \ref{bsdrank0} and multiplying both sides by $c_{\infty}(A) \cdot m_A$ we find the following conjectural equality
\begin{align*}
    c_{\infty}(A) \cdot m_A \cdot |A(\mathbb{Q})_{\rm{tors}}| \cdot \frac{L(A,1)}{\Omega(A)} \cdot |A^{\vee}(\mathbb{Q})_{\textrm{tors}}| = c_{\infty}(A) \cdot m_A \cdot  |\Sha(A/\mathbb{Q})|\cdot \prod_{p} c_p(A).
\end{align*}

 Since $c_{\infty}(A) \cdot m_A \cdot |A(\mathbb{Q})_{\rm{tors}}| \cdot \frac{L(A,1)}{\Omega(A)} \in \mathbb{Z}$, using the above conjectural equality, we find that Conjecture \ref{bsdrank0} implies that $|A^{\vee}(\mathbb{Q})_{\textrm{tors}}|$ divides $c_{\infty}(A) \cdot m_A \cdot  |\Sha(A/\mathbb{Q})|\cdot \prod_{p} c_p(A)$. The first to notice this conjectural divisibility statement were Agashe and Stein in \cite[Page 468]{agashestein}. We state this divisibility as a conjecture below.

\begin{conjecture}\label{conjagashestein}
 Let $A/\mathbb{Q}$ be an abelian variety with $L(A,1)\neq 0$ which is an optimal quotient of $J_0(N)/\mathbb{Q}$ attached to a newform. Then 
    \begin{align*}
        |A^{\vee}(\mathbb{Q})_{\textrm{tors}}| \text{ divides } c_{\infty}(A) \cdot m_A \cdot  |\Sha(A/\mathbb{Q})|\cdot \prod_{p} c_p(A).
    \end{align*}
\end{conjecture}

In this article we are interested in Conjecture \ref{conjagashestein} when  $A/\mathbb{Q}$ has dimension $1$, i.e., $A=E$ is an elliptic curve. In this case, $A^{\vee}/\mathbb{Q}$ is canonically isomorphic to $E/\mathbb{Q}$. Under our assumptions on $A/\mathbb{Q}$, Manin's conjecture (see \cite[Conjecture 2.1]{agasheribetstein}) predicts that $m_A=1$. If $E/\mathbb{Q}$ is in addition semi-stable, then it follows from work of Česnavičius (see \cite{cesnaviciusmaninsemistable}) that $m_E=1$. 

Our first Theorem confirms Conjecture \ref{conjagashestein} when $E/\mathbb{Q}$ is semi-stable.

\begin{theorem}\label{maintheoremsemistable}
Let $E/\mathbb{Q}$ be a semi-stable optimal elliptic curve such that $L(E,1) \neq 0$. Then $|E(\mathbb{Q})_{\rm{tors}}|$ divides $c_{\infty}(E) \cdot |\Sha(E/\mathbb{Q})|\cdot \prod_{p} c_p(E)$.
\end{theorem}

Our strategy for proving Theorem \ref{maintheoremsemistable} is the following. By a celebrated theorem of Mazur (see \cite[Theorem (8)]{maz}) there is a classification of all the possible rational torsion subgroups of rational elliptic curves. Using this classification, we perform a case by case analysis using explicit equations of elliptic curves with torsion points. The hardest cases to handle are the cases where the torsion subgroup is isomorphic to $\mathbb{Z}/2\mathbb{Z}$ or $\mathbb{Z}/3\mathbb{Z}$. Along the way we prove the following theorem, which does not require the assumption that $E/\mathbb{Q}$ is optimal.

\begin{theorem}\label{maintheorem}
Let $E/\mathbb{Q}$ be an elliptic curve.
\begin{enumerate}
    \item If $E(\mathbb{Q})_{\rm{tors}} \ncong \mathbb{Z}/2\mathbb{Z}, \mathbb{Z}/3\mathbb{Z}$, then $|E(\mathbb{Q})_{\rm{tors}}|$ divides $c_{\infty}(E) \cdot \prod_{p} c_p(E)$ except for the curves with Cremona \cite{cremonabook} label 11a3, 14a4, 14a6, 15a3, 15a7, 15a8, 17a2, 17a4, 20a2, 21a4, 24a4, and 32a2.
    \item Assume that $E(\mathbb{Q})_{\rm{tors}} \cong \mathbb{Z}/2\mathbb{Z}$. If $E/\mathbb{Q}$ is semi-stable away from $2$, then assume also that $E/\mathbb{Q}$ has semi-stable reduction modulo $2$. Then $|E(\mathbb{Q})_{\rm{tors}}|$ divides $c_{\infty}(E) \cdot \prod_{p} c_p(E)$, except for the curves with Cremona \cite{cremonabook} label 15a8, 39a4, and 55a4.
    \item Assume that $E/\mathbb{Q}$ has analytic rank $0$ with $E(\mathbb{Q})_{\rm{tors}} \cong \mathbb{Z}/3\mathbb{Z}$ and that $j_E \neq 0,1728$, where $j_E$ is the $j$-invariant of $E/\mathbb{Q}$. Then $|E(\mathbb{Q})_{\rm{tors}}|$ divides $\prod_{p} c_p(E) \cdot |\Sha(E/\mathbb{Q})|$, except possibly if $E/\mathbb{Q}$ is an elliptic curve which satisfies either $(a)$ or $(b)$ below.
    \begin{enumerate}[label=(\alph*)]
       \item $E/\mathbb{Q}$ is semi-stable away from 3, has at most one place of split multiplicative reduction, and the local root number $w_3(E)$ of $E/\mathbb{Q}$ at $3$ is equal to $1$.
       \item $E/\mathbb{Q}$ is semi-stable away from 3 and has reduction type II or IV modulo $3$.
     \end{enumerate}
\end{enumerate}
\end{theorem}

\begin{remark}
Since in each isogeny class of the Cremona database the curve with $1$ at the end of its label is the optimal one (with one exception; the curve with label 990H3 is optimal), we see that none of the exceptions of Parts $(i)$ and $(ii)$ of Theorem \ref{maintheorem} is optimal. Therefore, Parts $(i)$ and $(ii)$ of Theorem \ref{maintheorem} already prove Theorem \ref{maintheoremsemistable} when $E(\mathbb{Q})_{\rm{tors}} \ncong \mathbb{Z}/3\mathbb{Z}$.
\end{remark}

In Example \ref{examplesection2} below we present an elliptic curve curve $E/\mathbb{Q}$ with $E(\mathbb{Q})_{\rm{tors}} \cong \mathbb{Z}/2\mathbb{Z}$ such that $2$ does not divide $c_{\infty}(E) \cdot \prod_{p} c_p(E)$. Thus the reduction assumption is necessary in Part $(ii)$ of Theorem \ref{maintheorem}. Moreover, in Example \ref{examplesection3} below we present (non-optimal) elliptic curves $E/\mathbb{Q}$ with $E(\mathbb{Q})_{\rm{tors}} \cong \mathbb{Z}/3\mathbb{Z}$ such that $\prod_{p} c_p(E) \cdot |\Sha(E/\mathbb{Q})|=1$.  Thus we have to take into account the exceptions $(a)$ and $(b)$ in Part $(iii)$ of Theorem \ref{maintheorem}.

Part $(i)$ of Theorem \ref{maintheorem} builds on earlier work of Lorenzini in \cite{lor}, who initiated the study of Tamagawa numbers of abelian varieties with torsion points, as well as work of Byeon, Kim, and Yhee in \cite{bky2}. Moreover, for the proof of Part $(iii)$ of Theorem \ref{maintheorem} a variation of a technique developed by the author in \cite{mentzelosagasheconjecture} is used.

This paper is organized as follows. In Section \ref{section2} we study the interaction between torsion points and Tamagawa numbers when $E(\mathbb{Q})_{\rm{tors}} \ncong \mathbb{Z}/3\mathbb{Z}$ and we prove Parts $(i)$ and $(ii)$ of Theorem \ref{maintheorem}. Section \ref{section3} is more technical and is devoted to the proof of Part $(iii)$ of Theorem \ref{maintheorem}. Finally, in the last section we consider semi-stable elliptic curves and we prove Theorem \ref{maintheoremsemistable}.

\begin{acknowledgement}
The author would like to thank Dino Lorenzini for pointing out that some of the author's techniques in \cite{mentzelosagasheconjecture} could be applied to Conjecture \ref{conjagashestein}. The author would also like to thank Paul Voutier for some useful comments on an earlier version of this manuscript. I would also like to thank the anonymous referee for many insightful comments and many useful suggestions that greatly improved the exposition of this manuscript. This work was performed at the Steklov International Mathematical Center, while the author was a member there, and supported by the Ministry of Science and Higher education of the Russian Federation (Agreement no. 075-15-2019-1614).
\end{acknowledgement}

\section{Case $E(\mathbb{Q})_{\rm{tors}}\not\cong \mathbb{Z}/3\mathbb{Z}$}\label{section2}

Let $E/\mathbb{Q}$ be an elliptic curve and let $p$ be a prime number. The set $E_0(\mathbb{Q}_p)$ consisting of points with nonsingular reduction is a finite index subgroup of $E(\mathbb{Q}_p)$. The number $c_{p}(E)=[E(\mathbb{Q}_p):E_0(\mathbb{Q}_p)]$ is called the Tamagawa number of $E/\mathbb{Q}$ at $p$. We define the (global) Tamagawa number of $E/\mathbb{Q}$ as $c(E):=\prod_{p} c_{p}(E)$, where the product is taken over all the primes of bad reduction of $E/\mathbb{Q}$. We also define the Tamagawa number at infinity $c_\infty (E)$ for the place at infinity as follows: $c_\infty (E)=1$ if $\Delta(E)<0$ and $c_\infty (E)=2$ if $\Delta(E)>0$. Note that $c_\infty (E)$ is the number of connected components of $E(\mathbb{R})$ (see \cite[Corollary V.2.3.1]{silverman2}).

Tate, in \cite{tatealgorithm}, has produced an algorithm that computes the Tamagawa number of an elliptic curve defined over $\mathbb{Q}_p$. We recall a small part of the algorithm that we will use in this section and we refer the reader to \cite[Section IV.9]{silverman2} (or \cite{tatealgorithm}) for more details. Let $E/\mathbb{Q}_p$ be an elliptic curve given by a Weierstrass equation $$y^2+a_1xy+a_3y=x^3+a_2x^2+a_4x+a_6,$$ with $a_i \in \mathbb{Z}_p$ for $i = 1,2,3,4,6,$ ord$_p(a_3)>0$, ord$_p (a_4)>0$, ord$_p (a_6) >0$, and such that ord$_p(c_4)=0$ and ord$_p(\Delta) > 0$. Here $\Delta$ is the discriminant and $c_4$ is the $c_4$-invariant of the Weierstrass equation. Set $b_2:=a_1^2+4a_2$. Since ord$_p(c_4) = 0$ and ord$_p(a_3)$, ord$_p (a_4),$ ord$_p (a_6) >0$, we obtain that ord$_p(b_2)=0$ (see \cite[Section III.1]{aec} for the standard equation involving $c_4$ and $b_2$). Let $k'$ be the splitting field over $\mathbb{F}_p$ of the polynomial $T^2+a_1T+a_2$. The curve $E/\mathbb{Q}_p$ has split multiplicative reduction of type I$_n$ if ord$_p(\Delta)=n$ and $k'=\mathbb{F}_p$.  In this case, the Tamagawa number $c_p(E)$ of $E/\mathbb{Q}_p$ is equal to $n$. We will use the following observations repeatedly without explicit mention in this section.

\underline{Observation 1:}\label{observation1} If ord$_p(a_2),$ ord$_p(a_3),$ ord$_p(a_4),$ ord$_p (a_6) >0,$ and ord$_p(a_1)=0$, then $E/\mathbb{Q}_p$ has split multiplicative reduction of type I$_n$, where ord$_p(\Delta)=n$. Moreover, in this case $n = c_p(E)$.

\underline{Observation 2:}
When $E/\mathbb{Q}_p$ has multiplicative but not split multiplicative reduction, then the Tamagawa number $c_p(E)$ is $2$ if ord$_p(\Delta)$ is even, and $1$ otherwise.

Let $E/\mathbb{Q}$ be an elliptic curve with a $\mathbb{Q}$-rational point of order $4$. In \cite[Remark 2.6]{lor} Lorenzini asked whether the statement $4$ divides $c(E)\cdot c_\infty (E)$ is true with only finitely many exceptions. We prove this statement below.

\begin{proposition}\label{proptorsion4}
Let $E/\mathbb{Q}$ be an elliptic curve with a rational point of order $4$. Then $4$ divides  $c(E)\cdot c_\infty (E)$, except for the curves with Cremona \cite{cremonabook} labels 15a7, 15a8, 17a4, 21a4, and 24a4.
\end{proposition}

\begin{proof}
Let $E/\mathbb{Q}$ is an elliptic curve with a $\mathbb{Q}$-rational point of order $4$. Then $E/\mathbb{Q}$ can be given by an equation of the form
$$E_\lambda:\quad y^2+xy-\lambda y=x^3-\lambda x^2,$$ with $\lambda \in \mathbb{Q}$ (see \cite[Section 4.4]{hus}).
The discriminant of this Weierstrass equation is $$\Delta=\lambda^4(16\lambda+1).$$

Write $\lambda=\frac{s}{t}$ where $s,t$ are coprime integers and $s>0$. Using the change of variables $x \xrightarrow{} \frac{x}{t^2}, \; y \xrightarrow{} \frac{y}{t^3}$ we obtain a Weierstrass equation of the form
$$ y^2+txy-st^2y=x^3-stx^2$$ with discriminant $$\Delta'=s^4t^7(16s+t).$$
Lorenzini (see \cite[Proposition 2.4]{lor}) has proved that $2 \mid c(E)$ except for the curves with Cremona labels 15a7, 15a8 and 17a4. Therefore, if $\Delta'>0$ then we obtain that $4 \mid c(E)c_\infty (E)$. Hence, we can assume that $\Delta'<0$ from now on.

Assume first that $s>1$. If a prime $p \mid s$, then it follows from the discussion before Proposition \ref{proptorsion4} that $E/\mathbb{Q}$ has split multiplicative reduction modulo $p$ with ord$_p(\Delta')=4$ord$_p(s)$ and, hence, $4 \mid c_p(E)$. Assume now that $s=1$ (note that $s>0$ by assumption). If $\Delta'=t^7(16+t)<0$, then $-16<t<0$. Therefore, since $t$ is an integer there are only finitely many possibilities for $t$. Using SAGE \cite{sagemath} and computing the corresponding curves we obtain two more exceptions with Cremona labels 24a4 and 21a4.
\end{proof}

\begin{proposition}\label{proptorsion2x6}
If $E/\mathbb{Q}$ be an elliptic curve with $E(\mathbb{Q})_{tors} \cong \mathbb{Z}/2\mathbb{Z}\oplus \mathbb{Z}/6\mathbb{Z}$, then $|E(\mathbb{Q})_{\rm{tors}}|$ divides $c(E)$.
\end{proposition}
\begin{proof}
Let $E/\mathbb{Q}$ be an elliptic curve with $E(\mathbb{Q})_{tors} \cong \mathbb{Z}/2\mathbb{Z}\oplus \mathbb{Z}/6\mathbb{Z}$. By \cite[Page 26]{rab} the curve $E/\mathbb{Q}$ can be given by a Weierstrass equation
\begin{align}\label{eq2x6}
    y^2+(-t^2+4t+1)xy-t(t-1)(t+1)^2(3t+1)y=x^3-t(t-1)(t+1)^2x^2,
\end{align}
for some $t \in \mathbb{Q}$.
The discriminant of this Weierstrass equation is $$\Delta=t^6(t-1)^6(t+1)^6(3t-1)^2(3t+1)^2.$$

If there exists a prime $p$ with either ord$_p(t)=m>0$ or ord$_p(t-1)=m>0$ or ord$_p(t+1)=m>0$, then we find that the reduction of $E/\mathbb{Q}$ modulo $p$ is split multiplicative of type I$_{6m}$ with $c_p(E)=6m$. Moreover, if there exists a prime $p$ such that ord$_p(t)=m<0$, then Equation $($\ref{eq2x6}$)$ is not integral at $p$, but the following equation is
\begin{align}
    y^2+(\mu^2+4\mu-1)xy-\mu(1-\mu)(1+\mu)^2(3+\mu)y=x^3-(1-\mu)(1+\mu)^2x^2,
\end{align} where $\mu = \frac{1}{t}$. The discriminant of the above equation is $$\Delta'=\mu^2(1-\mu)^6(1+\mu)^6(3-\mu)^2(3+\mu)^2$$ and $$c_4'=(3+\mu^2)(3+75\mu^2-15\mu^4+\mu^6).$$
Therefore, if $p \neq 3$, since ord$_p(c_4')=0$, then $E/\mathbb{Q}$ has multiplicative reduction at $p$ with $2 \mid c_p(E)$ because ord$_p(\Delta')$ is even. Finally, when $p=3$ and ord$_3(t)<-1$, which implies that ord$_3(\mu)>1$, we find that ord$_3(c_4')=2$ and that ord$_3(\Delta')>6$. Therefore, using \cite[Tableau II]{pap} we obtain that $E/\mathbb{Q}$ has modulo $3$ reduction of type I$_n^*$ for some $n \geq 1$ and, hence, $2 \mid c_3(E)$.

 We now use the observations of the previous paragraph to prove that $12$ divides $c(E)$. Write $t=\frac{a}{b}$, with $a,b \in \mathbb{Z}$ coprime and $b>0$. We then have that $t(t-1)(t+1)=\frac{a}{b}(\frac{a}{b}-1)(\frac{a}{b}+1)=\frac{a(a-b)(a+b)}{b^2}$. If $a$ has two or more prime divisors, then $6^2 \mid c(E)$ and our proposition is proved. Therefore, we can assume from now on that $a$ has at most one prime divisor. We split the proof into two cases, depending on whether $|a|$ is equal to $1$ or to a power of a prime.

Assume first that $a = \pm p^m$, for some prime $p$ and some $m>0$. We already have that $6 \mid c_p(E)$. If $b \neq 1$ or $3$, then $2 \mid c_r(E)$ for some divisor $r$ of $b$ because ord$_r(t)<0$. Therefore, we can assume that $b=1$ or $3$. We will show that $(a-b)(a+b)$ has at least one odd prime divisor $q$ and, hence, $6 \mid c_q(E)$. If $b=1$, then $(a-b)(a+b)=(a-1)(a+1)$, which cannot be a power of $2$. Therefore, we find that $(a-b)(a+b)$ has at least one odd prime divisor $q$, except for $a= \pm 3$. Consequently, we have that $6 \mid c_q(E)$, except possibly for $a= \pm 3$. If $b=3$, then $(a-b)(a+b)=(a-3)(a+3)$, but since $a-3$ and $a+3$ differ by $6$, they cannot both be a power of $2$, except for $a=\pm 5$. Therefore, there exists some odd prime $q$ that divides $(a-b)(a+b)$ and, hence, $6 \mid c_q(E)$, except possibly when $a=\pm 5$. Thus, combining all the cases we have proved that if $a = \pm p^m$ for some prime $p$ and some $m>0$, then $12 \mid c(E)$ except possibly for the curves that correspond to $(a,b)=(\pm 3, 1), (\pm5, 3)$.

Assume now that $a=\pm 1$. Then $(a-b)(a+b)=a^2-b^2=1-b^2=(1-b)(1+b)$. Since $1-b$ and $1+b$ differ by $2$, they cannot both be a power of $2$, except for $b= 3$. Therefore, there exists some odd prime $q$ that divides $(a-b)(a+b)$, which implies that $6 \mid c_q(E)$, except for $b= 3$. On the other hand, since $b \neq 3$, then $2 \mid c_r(E)$ for some divisor $r$ of $b$. Thus, combining all the cases we have proved that for $a = \pm 1$, then  $12 \mid c(E)$ except possibly for the curve that corresponds to $(a,b)=(\pm 1,  3)$.

Finally, the exceptions $t=\pm 3, \pm \frac{1}{3}, \pm \frac{5}{3}$ correspond to either singular curves or the elliptic curves with Cremona label 30a2 and 90c6, which have Tamagawa number divisible by $12$. This concludes our proof.
\end{proof}

We are now ready to proceed to the proof of Part $(i)$ of Theorem \ref{maintheorem}.

\begin{proof}[Proof of Part $(i)$ Theorem \ref{maintheorem}]
By a Theorem of Mazur (see \cite[Theorem (8)]{maz}) we know that if $E/\mathbb{Q}$ is an elliptic curve over $\mathbb{Q}$, then $E(\mathbb{Q})_{\rm{tors}}$ is isomorphic to
$$\mathbb{Z}/N\mathbb{Z} \text{ for }N=1,2,...,10 \text{ or } \mathbb{Z}/2\mathbb{Z} \oplus \mathbb{Z}/2N\mathbb{Z} \text{ for }N=1,2,3,4.$$

Lorenzini in \cite[Proposition 1.1]{lor} and Byeon, Kim, and Yhee in \cite[Proposition 3.1 and Proposition 3.2]{bky2} have proved the following results.

\begin{theorem}\label{thmlorenzinibyeonkimyhee}
Let $E/\mathbb{Q}$ be a elliptic curve.
\begin{enumerate}
    \item \textup{(Lorenzini)}. If $E(\mathbb{Q})_{\rm{tors}} \in \{ \mathbb{Z}/5\mathbb{Z},  \mathbb{Z}/6\mathbb{Z}, \mathbb{Z}/7\mathbb{Z}, \mathbb{Z}/8\mathbb{Z}, \mathbb{Z}/9\mathbb{Z}, \mathbb{Z}/10\mathbb{Z}, \mathbb{Z}/12\mathbb{Z}, \mathbb{Z}/2\mathbb{Z} \oplus \mathbb{Z}/8\mathbb{Z}\}$, then $|E(\mathbb{Q})_{\rm{tors}}|$ divides $c(E)$ except for the curves with Cremona label 11a3, 14a4, 14a6, and 20a2.
    \item \textup{(Byeon, Kim, and Yhee)}. If $E(\mathbb{Q})_{\rm{tors}} \cong \mathbb{Z}/2\mathbb{Z} \oplus \mathbb{Z}/2\mathbb{Z} $, then $|E(\mathbb{Q})_{\rm{tors}}|$ divides $c(E)$ except for the curves with Cremona label 17a2 and 32a2.
    \item \textup{(Byeon, Kim, and Yhee)}. If $E(\mathbb{Q})_{\rm{tors}} \cong \mathbb{Z}/2\mathbb{Z} \oplus \mathbb{Z}/4\mathbb{Z} $, then $|E(\mathbb{Q})_{\rm{tors}}|$ divides $c(E)$ except for the curve with Cremona label 15a3.
\end{enumerate}
\end{theorem}

If $E(\mathbb{Q})_{\rm{tors}} \ncong \mathbb{Z}/4\mathbb{Z}, \mathbb{Z}/2\mathbb{Z}\oplus \mathbb{Z}/6\mathbb{Z}$ then Theorem \ref{thmlorenzinibyeonkimyhee} shows $|E(\mathbb{Q})_{\rm{tors}}|$ divides $c(E)$ except for the curves with Cremona label 11a3, 14a4, 14a6, 15a3, 17a2, 20a2, and 32a2. Moreover, Proposition \ref{proptorsion2x6} implies that if $E(\mathbb{Q})_{\rm{tors}} \cong \mathbb{Z}/2\mathbb{Z}\oplus \mathbb{Z}/6\mathbb{Z}$, then $|E(\mathbb{Q})_{\rm{tors}}|$ divides $c(E)$. Finally, if $E(\mathbb{Q})_{\rm{tors}} \cong \mathbb{Z}/4\mathbb{Z}$, then using Proposition \ref{proptorsion4}, we find that $E(\mathbb{Q})_{\rm{tors}}$ divides $c(E) \cdot c_\infty (E)$, except for the curves with Cremona \cite{cremonabook} labels 15a7, 15a8, 17a4, 21a4, and 24a4. This proves Part $(i)$ of Theorem \ref{maintheorem}.

\end{proof}

Before proceed to the proof of Part $(ii)$ of Theorem \ref{maintheorem} we need the following Propositions.

\begin{proposition}\label{prop2torsionsemistable}
Let $E/\mathbb{Q}$ be a semi-stable elliptic curve with $E(\mathbb{Q})_{\rm{tors}} \cong \mathbb{Z}/2\mathbb{Z}$. Then $|E(\mathbb{Q})_{\rm{tors}}|$ divides $c_{\infty}(E) \cdot c(E)$, except for the curves with Cremona labels 15a8, 39a4, and 55a4.
\end{proposition}
\begin{proof}
 Let $E/\mathbb{Q}$ be a semi-stable elliptic curve with $E(\mathbb{Q})_{\rm{tors}} \cong \mathbb{Z}/2\mathbb{Z}$. By \cite[Lemme 1]{mestreoesterle1989} the curve $E/\mathbb{Q}$ has a Weierstrass equation of the form $$y^2=x^3+ax^2+bx,$$ with $a,b \in \mathbb{Z}$ and $a,b$ coprime. This equation is minimal outside $2$ (see \cite[Page 176]{mestreoesterle1989}), with discriminant $$\Delta=16b^2(a^2-4b) \text{ and } c_4=16(a^2-3b).$$ Moreover, the minimal discriminant of $E/\mathbb{Q}$ is $\Delta'=\frac{1}{2^8}b^2(a^2-4b)$.
 
 If $p$ is an odd prime with $p \mid b$, then since ord$_p(c_4)=0$ and ord$_p(\Delta)$ is even, the curve $E/\mathbb{Q}$ has multiplicative reduction modulo $p$ with $2 \mid c_p(E)$. Therefore, we can assume from now on that $b=\pm 2^u$, for some $u \geq 0$. Moreover, if $u \geq 5$, then ord$_2(\Delta')>0$ and it is even and, hence, $E/\mathbb{Q}$ has multiplicative reduction modulo $2$ with $2 \mid c_2(E)$. Consequently, we can assume that $4 \geq u \geq 0$.
 
 Now, if $a^2-4b>0$, then $c_{\infty}(E)=2$. So, we can also assume that $a^2-4b<0$. This implies that $b>0$ and as a result the possible options for $b$ are $1,2,4,8,$ or $16$. Moreover, since we require that $a^2-4b<0$ and we have that $a^2 \geq 0$, we obtain only finitely many possibilities for $a$. After computations we find that $a$ can only be equal to $ 0, \pm 1,\pm 3, \pm 5, \pm 7$. Using SAGE and computing the curves corresponding to the above choices of $a$ and $b$ we find the exceptions with Cremona labels 15a8, 39a4, 55a4.
\end{proof}

\begin{proposition}\label{2torsionprimeofadditivereduction}
Let $E/\mathbb{Q}$ be an elliptic curve with $E(\mathbb{Q})_{\rm{tors}} \cong \mathbb{Z}/2\mathbb{Z}$ and such that it has a prime $p \neq 2$ of additive reduction. Then $|E(\mathbb{Q})_{\rm{tors}}|$ divides $c_p(E)$.
\end{proposition}
\begin{proof}
Suppose that $E/\mathbb{Q}$ is given by a minimal Weierstrass equation at $p$ and let $\widetilde{E}_{\rm{ns}}(\mathbb{F}_p)$ be the group of nonsigular points of the reduction of $E/\mathbb{Q}$ at $p$. There is a short exact sequence $$0 \longrightarrow E_1(\mathbb{Q}_p) \longrightarrow E_0(\mathbb{Q}_p) \longrightarrow \widetilde{E}_{\rm{ns}}(\mathbb{F}_p)\longrightarrow 0,$$
 where $E_0(\mathbb{Q}_p)$ consists of all points with nonsingular reduction and $E_1(\mathbb{Q}_p)$ is the kernel of the reduction map. Let $P$ be a $\mathbb{Q}$-rational point of order $2$. We  claim that $P \in E(\mathbb{Q}_p)/ E_0(\mathbb{Q}_p)$, so $2 \mid c_p(E)$ because $c_p(E)= |E(\mathbb{Q}_p)/ E_0(\mathbb{Q}_p)|$. Assume that $P \in E_0(\mathbb{Q}_p)$ and we will find a contradiction. Since $p \neq 2$, it follows that $E_1(\mathbb{Q}_p)$ cannot have points of order $2$ (see \cite[Proposition VII.3.1]{aec}) so $P \not\in E_1(\mathbb{Q}_p)$. This implies that the point $P$ reduces to a point of order $2$ in $\widetilde{E}_{\rm{ns}}(\mathbb{F}_p)$. This is a contradiction because $E/\mathbb{Q}$ has additive reduction modulo $p$ and, hence, $\widetilde{E}_{\rm{ns}}(\mathbb{F}_p) \cong (\mathbb{F}_p,+)$, which does not contain points of order $2$ because $p\neq 2$.
\end{proof}

\begin{proof}[Proof of Part $(ii)$ of Theorem \ref{maintheorem}]
If $E/\mathbb{Q}$ is semi-stable away from $2$, then by our assumption $E/\mathbb{Q}$ is semi-stable and, hence, Proposition \ref{prop2torsionsemistable} implies that $|E(\mathbb{Q})_{\rm{tors}}|$ divides $c_{\infty}(E)\cdot c(E)$, except for the curves with Cremona labels 15a8, 39a4, and 55a4. If $E/\mathbb{Q}$ is not semi-stable away from $2$, then it has a prime $p$ of additive reduction. Using Proposition \ref{2torsionprimeofadditivereduction} we obtain that $|E(\mathbb{Q})_{\rm{tors}}|$ divides $c_p(E)$. This proves Part $(ii)$ of Theorem \ref{maintheorem}.
\end{proof}

\begin{example}\label{examplesection2}
Let  $E/\mathbb{Q}$ be the elliptic curve with Cremona \cite{cremonabook} label 48a4 (LMFDB \cite{lmfdb} label 48.a5). Then $E(\mathbb{Q}) \cong \mathbb{Z}/2\mathbb{Z}$, $|\Sha(E/\mathbb{Q})|=1$, and $c_{\infty}(E) \cdot c(E)=1$. Therefore, $|E(\mathbb{Q})_{\rm{tors}}|$ does not divide $|\Sha(E/\mathbb{Q})|\cdot c_{\infty}(E) \cdot c(E)$. Note however, that $E/\mathbb{Q}$ is not optimal and has Manin constant equal to $2$.
\end{example}

\section{Case $E(\mathbb{Q})_{\rm{tors}} \cong \mathbb{Z}/3\mathbb{Z}$}\label{section3}

In this section we prove Part $(iii)$ of Theorem \ref{maintheorem} using ideas from \cite{mentzelosagasheconjecture}. First we recall some background material on reduction types of elliptic curves. We refer the reader to \cite[Chapter IV]{silverman2} for more information. If $E/\mathbb{Q}$ is an elliptic curve and $p$ is a prime, then we denote by $\mathcal{E}^{min}/\mathbb{Q}_p^{unr}$ the minimal proper regular model of $E/\mathbb{Q}$ over the maximal unramified extension $\mathbb{Q}_p^{unr}$ of $\mathbb{Q}_p$. There is a classification of the different configurations for the special fiber of $\mathcal{E}^{min}/\mathbb{Q}_p^{unr}$ by work Kodaira and N\'eron. More precisely, the special fiber of $\mathcal{E}^{min}/\mathbb{Q}_p^{unr}$ belongs to one of the following types; II, II$^*$, III, III$^*$, IV, IV$^*$, I$_n$ for $n \in \mathbb{Z}$ with $n \geq 0$, or I$_n^*$ for $n \in \mathbb{Z}$ with $n \geq 0$ (see \cite[Theorem 8.2]{silverman2} for the geometric meaning of those types). If $\mathcal{E}^{min}/\mathbb{Q}_p^{unr}$ is of type $T$ that belongs in the list above, then we will say that $E/\mathbb{Q}$ has reduction type $T$ modulo $p$.

We now restate Part $(iii)$ of Theorem \ref{maintheorem} for the convenience of the reader.

\begin{theorem}\label{exceptionlist}
Let $E/\mathbb{Q}$ be an elliptic curve with a $\mathbb{Q}$-rational point of order $3$ and $j$-invariant $j_E \neq 0,1728$. Assume that the analytic rank of $E/\mathbb{Q}$ is $0$. Then $3$ divides $c(E)\cdot |\Sha(E/\mathbb{Q})|$, except possibly if $E/\mathbb{Q}$ is an elliptic curve which satisfies either $(a)$ or $(b)$ below.
\begin{enumerate}[label=(\alph*)]
    \item $E/\mathbb{Q}$ is semi-stable away from 3, has at most one place of split multiplicative reduction, and the local root number $w_3(E)$ of $E/\mathbb{Q}$ at $3$ is equal to $1$.
    \item $E/\mathbb{Q}$ is semi-stable away from 3 and has reduction type II or IV modulo $3$.
\end{enumerate}
\end{theorem}

Before we proceed to our proof we give a brief overview so that it is easier for the reader to follow. The idea is that we first examine when $3$ divides $c(E)$ using the reduction properties of $E/\mathbb{Q}$. When this divisibility fails we prove that $9$ divides $|\Sha(E/\mathbb{Q})|$. More precisely, the strategy of the proof is as follows. Since $E/\mathbb{Q}$ has a $\mathbb{Q}$-rational point of order $3$ we can find a very simple Weierstrass equation for $E/\mathbb{Q}$. Invoking results of Kozuma in Proposition \ref{prop:3torsionreduction} we find, in Lemma \ref{lemmatamagawanumber} and the paragraph after it, that if $3$ does not divide $c(E)$, then $E/\mathbb{Q}$ must be of a very specific form. For the case where $3$ does not divide $c(E)$, we need to show that $9$ divides $|\Sha(E/\mathbb{Q})|$. We achieve this using Lemma \ref{tamagawagreaterorequalto2}, which takes advantage of an isogeny $\phi:  E \rightarrow{} \widehat{E}$ coming form the point of order $3$. However, in order to apply Lemma \ref{tamagawagreaterorequalto2} we need to prove that $\text{ord}_3(\frac{\prod_{p}c_p(\widehat{E})}{\prod_{p}c_p(E)}) \geq 2$. This is achieved in Claim \ref{lastclaimproof} using Claim \ref{splitmultiplicativereduction}.

\begin{proof}

Since the analytic rank of $E/\mathbb{Q}$ is zero, work of Gross and Zagier, on heights of Heegner points \cite{grosszagierpaper}, as well as work of Kolyvagin, on Euler systems \cite{kolyvagineulersystems}, imply that  $E/\mathbb{Q}$ has (algebraic) rank $0$ and that $\Sha(E/\mathbb{Q})$ is finite (see \cite[Theorem $3.22$]{darmonmodularellipticcurves} for a sketch of the proof). Thus the statement of Theorem \ref{exceptionlist} makes sense.

If $E/\mathbb{Q}$ be an elliptic curve with a $\mathbb{Q}$-rational point $P$ of order $3$, then by translating $P$ to $(0,0)$ we find a Weierstrass equation of the form $$y^2+cxy+dy=x^3,$$ with $c,d \in \mathbb{Q}$ (see \cite[Remark $2.2$ in Section $4.2$]{hus}). If $u \in \mathbb{Z}$, then the transformation $(x,y) \rightarrow{(\frac{x}{u^2}, \frac{y}{u^3})}$ gives a new Weierstrass equation of the same form with $c$ replaced by $uc$ and with $d$ replaced by $u^3d$ (see \cite[Page 185]{aec}). Therefore, by picking $u$ to be the product of appropriate powers of the primes appearing in the denominators of $c,d$ (if any), we can arrange that $c,d \in \mathbb{Z}$. Moreover, by using the transformation $(x,y) \rightarrow{(x, -y)}$ if necessary, we can arrange that $d > 0$. We now show that we can find a new Weierstrass equation of the above form with coefficients, which we will call $a$ and $b$ below, such that for every prime $q$ either $q \nmid a$ or $q^3 \nmid b$. If there is no prime $q$ such that $q \mid c$ and $q^3 \mid d$, then set $a=c$ and $b=d$. Otherwise, let $q_1,...,q_s$ be the set of primes such that $q_i \mid c$ and ${q_i}^3 \mid d$, and let $n_i= \mathrm{min}\{ \mathrm{ord}_{q_i}(c), \lfloor \frac{\mathrm{ord}_{q_i}(d)}{3} \rfloor \} $, where $\lfloor -  \rfloor$ is the floor function. If $u=\prod_{i=1}^{s} q_i^{n_i}$, then using the transformation $(x,y) \rightarrow{(u^2x, u^3y)}$ we obtain a new equation of the form $$y^2+\frac{c}{u}xy+\frac{d}{u^3}y=x^3.$$ Setting $a=\frac{c}{u}$ and $b=\frac{d}{u^3}$ we see that $a,b \in \mathbb{Z}$, $b >0$, and for every prime  $q$ either $q \nmid a$ or $q^3 \nmid b$. Therefore, we have proved that we can choose a Weierstrass equation for $E/\mathbb{Q}$ of the form
\begin{align}\label{eq:3torsion}
    y^2+axy+by=x^3,
\end{align}
where $a,b$ are integers, $b>0$, and for every prime $q$ either $q \nmid a$ or $q^3 \nmid b$. Also, we must have $a^3-27b \neq 0$, since the discriminant of Equation $($\ref{eq:3torsion}$)$ is $$\Delta=b^3(a^3-27b) \quad \text{and we also have} \quad c_4=a(a^3-24b).$$

The following proposition will be used in Lemma \ref{lemmatamagawanumber} below. The arrows in each part point towards the reduction type of $E/\mathbb{Q}$ at the prime $p$.
\begin{proposition}\label{prop:3torsionreduction} (See \cite[Proposition 3.5 and Lemma 3.6]{koz})
Let $E/\mathbb{Q}$ be an elliptic curve given by Equation $($\ref{eq:3torsion}$)$. Write $D:=a^3-27b$ and let $p$ be any prime (note that either $p \nmid a$ or $p^3 \nmid b$). Then the reduction of $E/\mathbb{Q}$ modulo $p$ is determined as follows: 
\begin{enumerate}
    \item $\text{If} \; 3ord_p(a) \leq ord_p(b)$ :  $\begin{cases} 
      3ord_p(a)< ord_p(b) \rightarrow \text{split I}_{3ord_p(b)}, \quad c_p(E)=3ord_p(b) \\
      3ord_p(a)= ord_p(b) \rightarrow \begin{cases} 
      ord_p(D)>0 \rightarrow \text{I}_{ord_p(D)}\\
      ord_p(D)=0 \rightarrow \text{Good reduction I}_0
   \end{cases}
\end{cases}$

    \item $\text{If} \; 3ord_p(a) > ord_p(b)$ :  $\begin{cases} 
      ord_p(b) = 0 \rightarrow \begin{cases} 
      p=3 \rightarrow \text{Go to} \; (iii) \\
      p \neq 3 \rightarrow \text{Good reduction I}_0
   \end{cases} \\
      ord_p(b) = 1  \rightarrow \text{IV, } c_p(E)=3\\
       ord_p(b) = 2 \rightarrow \text{IV}^*,\text{ } c_p(E)=3.
\end{cases}$

   \item If $p=3$ and $ord_p(a)>0=ord_p(b)$: $ord_p(D)=$ $\begin{cases}
3 \rightarrow \text{II or III}\\
4 \rightarrow \text{II}\\
5 \rightarrow \text{IV}\\
n \rightarrow \text{I}_{n-6}^* \text{  , for  } n \geq 6.
\end{cases}$

\end{enumerate}
\end{proposition}

\begin{lemma}\label{lemmatamagawanumber}
Let $E/\mathbb{Q}$ be an elliptic curve given by a Weierstrass equation of the form $($\ref{eq:3torsion}$)$. If $3 \nmid \prod_{p}c_p(E)$, then $b=1$.
\end{lemma}
\begin{proof}
If a prime $r$ divides $b$, then, by Proposition \ref{prop:3torsionreduction}, Parts $(i)$ and $(ii)$, $E/\mathbb{Q}$ has reduction type IV, IV$^*$ or split I$_{3\text{ord}_{r}(b)}$, and in any of these cases $3 \mid c_r(E)$.
\end{proof}

We assume from now on that $3 \nmid \prod_{p}c_p(E)$ and, hence, $b=1$ by Lemma \ref{lemmatamagawanumber}. We note that since $b=1$, we obtain from Proposition \ref{prop:3torsionreduction} that $E/\mathbb{Q}$ cannot have reduction of type IV$^*$ modulo $3$ and Proposition \ref{prop:3torsionreduction} implies that $E/\mathbb{Q}$ is semi-stable away from $3$.

Since $b=1$, we have that $E/\mathbb{Q}$ has a Weierstrass equation of the form
\begin{align}\label{eq:3torsion2}
    y^2+axy+y=x^3
\end{align}
where $a$ is an integer. We must have $a^3-27 \neq 0$, since the discriminant of this Weierstrass equation is $$\Delta=a^3-27 \quad \text{and we also have} \quad c_4=a(a^3-24).$$

Recall that essentially by the construction of our Weierstrass equation the point $(0,0)$ has order $3$ and denote by $<(0,0)>$ the group generated by the point $(0,0)$. Let $\widehat{E}:=E/<(0,0)>$ and let $\phi:  E \rightarrow{} \widehat{E}$ be the associated $3$-isogeny. We denote by $\hat{\phi}: \widehat{E} \rightarrow{} E$ the dual isogeny. The following lemma will be used repeatedly in what follows.

\begin{lemma}\label{tamagawagreaterorequalto2}
Let $E/\mathbb{Q}$ be an elliptic curve given by a Weierstrass equation of the form $($\ref{eq:3torsion2}$)$ and assume that $E/\mathbb{Q}$ has analytic rank $0$. Let $\widehat{E}:=E/<(0,0)>$ and let $\phi:  E \rightarrow{} \widehat{E}$ be the associated $3$-isogeny. If $\mathrm{ord}_3(\frac{\prod_{p}c_p(\widehat{E})}{\prod_{p}c_p(E)})\geq 2$, then $9$ divides $|\Sha(E/\mathbb{Q})|$.

\end{lemma}
\begin{proof}
The proof of Lemma \ref{tamagawagreaterorequalto2} is exactly the same as the proof of \cite[Lemma 3.6, Part $(i)$]{mentzelosagasheconjecture}. The proof of \cite[Lemma 3.6]{mentzelosagasheconjecture} (as well as the discussion after the proof of \cite[Claim 3.2]{mentzelosagasheconjecture}) is independent of the assumption in \cite[Theorem 3.1]{mentzelosagasheconjecture} that the curve has reduction of type I$_n^*$, for some $n \geq 0$.
\end{proof}

By a proposition of Hadano, see \cite[Theorem 1.1]{had}, since $b=1$, $\widehat{E}/\mathbb{Q}$ is given by $$y^2+(a+6)xy+(a^2+3a+9)y=x^3.$$ The discriminant of this equation is $\widehat{\Delta}=(a^3 - 27)^3$ and $\widehat{c_4}=a(a^3+216)$.

\begin{claim}\label{splitmultiplicativereduction}
The curve $\widehat{E}/\mathbb{Q}$ has a prime $q \neq 3$  of split multiplicative reduction except if $E/\mathbb{Q}$ has Cremona label 27a3, 27a4, or 54a3.
\end{claim}
\begin{proof}
If $q$ is a prime such that $q \mid a^2+3a+9$ and $q \mid a+6$, then $$q \mid a^2+3a+9-(a+6)^2=a^2+3a+9-a^2-12a-36=-9(a+3).$$ Therefore, either $q=3$ or $q \mid a+3$. If $q \mid a+3$, then since $q \mid a + 6$, we obtain that $q =3$. This proves that if $q \neq 3$ is a prime divisor of $a^2+3a+9$, then $\widehat{E}/\mathbb{Q}$ has split multiplicative reduction modulo $q$. 

First, note that $a^2+3a+9 >1$ so $a^2+3a+9$ cannot be equal to $\pm 1$. We now prove that $a^2+3a+9$ is not equal to a power of $3$ except for $a=0, \pm 3, -6$. Assume that $a^2+3a+9$ is a power of $3$. Then $3 \mid a$ so we can write $a=3a'$ for some $a'$. Therefore, we have that $$a^2+3a+9=9((a')^2+a'+1)$$ and $(a')^2+a'+1$ cannot be a power of $3$ unless $a'=1,-2$ because it is always non zero modulo $9$. Moreover, if $(a')^2+a'+1= \pm 1$, then $a'=0$ or $-1$. This proves that $a^2+3a+9$ is not equal to a power of $3$ except for $a=0, \pm 3, -6$.

If $a^2+3a+9$ is not a power of $3$, then $a^2+3a+9$ has a divisor $q \neq 3$ and, hence, $\widehat{E}/\mathbb{Q}$ has split multiplicative reduction modulo $q$. The exceptions $a=0, \pm 3, -6$ give the elliptic curves $E/\mathbb{Q}$ with Cremona labels 27a3, 27a4, or 54a3.
\end{proof}

Since $\widehat{E}/\mathbb{Q}$ has split multiplicative modulo $q$, we get that $c_q(\widehat{E})=\mathrm{ord}_q((a^3-27)^3)=3\mathrm{ord}_q(a^3-27)$. Moreover, since $E/\mathbb{Q}$ is isogenous to $\widehat{E}/\mathbb{Q}$ and $\widehat{E}/\mathbb{Q}$ has multiplicative reduction modulo $q$, by \cite[Theorem $5.4$ Part $(4)$]{dd}, we know that $E/\mathbb{Q}$ has multiplicative reduction modulo $q$. If $E/\mathbb{Q}$ has nonsplit multiplicative reduction modulo $q$, then $c_q(E)=1$ or $2$, depending on whether $\mathrm{ord}_q(\Delta)$ is odd or even, respectively. If $E/\mathbb{Q}$ has split multiplicative reduction modulo $q$, then $c_q(E)=\mathrm{ord}_q(\Delta)$. Therefore, $c_q(E) \mid \mathrm{ord}_q(\Delta)$. However, $\Delta=a^3-27$ and, hence,  $c_q(E) \mid \mathrm{ord}_q(a^3-27)$ which implies that $\textrm{ord}_3 (\frac{c_q(\widehat{E})}{c_q(E)})>0$.

\begin{emptyremark}\label{rootnumbers}
We now recall some facts concerning root numbers that will be needed in the proof of Claim \ref{lastclaimproof} below. Let $w(E):=\displaystyle \prod_{p \in M_{\mathbb{Q}} } w_p(E)$ where $w_p(E)$ is the local root number of $E/\mathbb{Q}$ at $p$ and $M_{\mathbb{Q}}$ is the set of places of $\mathbb{Q}$. The number $w(E)$ is called the global root number of $E/\mathbb{Q}$.
Since the analytic rank of $E/\mathbb{Q}$ is $0$, using \cite[Theorem 3.22]{darmonmodularellipticcurves}, we find that $\Sha(E/\mathbb{Q})$ is finite and the rank of $E/\mathbb{Q}$ is $0$. Consequently, by \cite[Theorem 1.4]{dd1} we obtain that $1=(-1)^{rk(E/\mathbb{Q})}=w(E)$. 
Recall that we assume that $j_E \neq 0, 1728$. The local root number $w_p(E)$ of $E/\mathbb{Q}$ at $p$ is as follows, see \cite[Page 95]{connel} and \cite[Page 132]{rohrlich} for $j_E \neq 0, 1728$ and $p \geq 5$, and \cite[Page 1051]{halberstadt} for $p=2$ or $3$,
\[ \begin{cases} 
      w_{\infty}(E)=-1 \\
      w_p(E)=1 & \mathrm{if} \ E/\mathbb{Q} \ \mathrm{has} \ \mathrm{modulo} \ p \ \mathrm{good} \ \mathrm{or} \ \mathrm{nonsplit} \ \mathrm{multiplicative} \ \mathrm{reduction}.\\
      w_p(E)=-1 & \mathrm{if} \ E/\mathbb{Q} \ \mathrm{has} \ \mathrm{modulo} \ p \ \mathrm{split} \ \mathrm{multiplicative} \  \mathrm{reduction}.\\
   \end{cases}
\]
\end{emptyremark}

 We will now show that $9$ divides $|\Sha(E/\mathbb{Q})$|, unless $E/\mathbb{Q}$ belongs to one of the exceptions to Theorem \ref{exceptionlist}. By Lemma \ref{tamagawagreaterorequalto2}, it is enough to show that $\text{ord}_3(\frac{\prod_{p}c_p(\widehat{E})}{\prod_{p}c_p(E)})\geq 2.$ We achieve this in the following claim.
\begin{claim}\label{lastclaimproof}
If $E/\mathbb{Q}$ does not belong to the families $(a)$ or $(b)$ of Theorem \ref{exceptionlist}, then $$\text{ord}_3(\frac{\prod_{p}c_p(\widehat{E})}{\prod_{p}c_p(E)}) \geq 2.$$
\end{claim}
\noindent {\it Proof.}
 Let $p \neq 3$ be any prime such that $E/\mathbb{Q}$ has multiplicative reduction modulo $p$. If $E/\mathbb{Q}$ has nonsplit multiplicative reduction modulo $p$, then by lines $5,6$, and $7$ of \cite[Theorem 6.1]{dd} we obtain that $\textrm{ord}_3(\frac{c_p(\widehat{E})}{c_p(E)})=0$. Note that in \cite[Theorem 6.1]{dd}, the curve $\widehat{E}$ is denoted by $E'$ and $\delta,\delta'$ are the valuations of the two discriminants. Since $E/\mathbb{Q}$ is semi-stable away from $3$, we find that  $\textrm{ord}_3(\frac{c_p(\widehat{E})}{c_p(E)})=0$ if $E/\mathbb{Q}$ does not have split multiplicative reduction modulo $p$. Assume now that $E/\mathbb{Q}$ has split multiplicative reduction modulo $p$. Recall that we assume that $b=1$, which implies that $\Delta=a^3-27$ and $\widehat{\Delta}=(a^3-27)^3$. If $\textrm{ord}_p(\Delta)=\gamma$, then $\textrm{ord}_p(\widehat{\Delta})=3\gamma$. Therefore, by line $3$ of \cite[Theorem 6.1]{dd} we obtain that $\textrm{ord}_3(\frac{c_p(\widehat{E})}{c_p(E)})=1$. Since $E/\mathbb{Q}$ is semi-stable away from $3$, the above arguments prove that if $p \neq 3$ is a prime, then $E/\mathbb{Q}$ has split multiplicative reduction modulo $p$ if and only if $\textrm{ord}_3(\frac{c_p(\widehat{E})}{c_p(E)})=1$.
 
 Assume that $E/\mathbb{Q}$ does not have reduction type II or IV modulo $3$ and that $w_3(E)=1$. We can also assume that $E/\mathbb{Q}$ has 2 or more places of split multiplicative reduction because otherwise we are in one of the exceptions to Theorem \ref{exceptionlist}. Line 13 of  \cite[Theorem 6.1]{dd} implies that $\textrm{ord}_3(\frac{c_{3}(\widehat{E})}{c_{3}(E)})=0$. If $p$ is a prime of split multiplicative reduction, then $\textrm{ord}_3(\frac{c_p(\widehat{E})}{c_p(E)})=1$. Moreover, if $r$ is a prime of nonsplit multiplicative reduction, then $\textrm{ord}_3(\frac{c_r(\widehat{E})}{c_r(E)})=0$. Since $E/\mathbb{Q}$ is semi-stable away from $3$, the above prove that $\text{ord}_3(\frac{\prod_{p}c_p(\widehat{E})}{\prod_{p}c_p(E)}) \geq 2.$
 
 Assume now that $E/\mathbb{Q}$ does not have reduction type II or IV modulo $3$, and that $w_3(E)=-1$. We claim that because the analytic rank of $E/\mathbb{Q}$ is zero, the curve $E/\mathbb{Q}$ has an even number of places of split multiplicative reduction. Indeed, by \ref{rootnumbers} we have that $w(E)=1$ and $w_{\infty}(E)=-1$. Moreover, we have that $E/\mathbb{Q}$ is semi-stable away from $3$, and by \ref{rootnumbers} for $p\neq 3$ we obtain that $w_p(E)=-1$ if and only if $E/\mathbb{Q}$ has split multiplicative reduction modulo $p$. This proves that $E/\mathbb{Q}$ has an even number of places of split multiplicative reduction. On the other hand, using the Cremona database \cite{cremonabook} it is easy to see that that the curves with Cremona label 27a3, 27a4, 54a3 all belong to the exceptional family $(b)$ of Theorem \ref{exceptionlist}. Therefore, Claim \ref{splitmultiplicativereduction} gives us that $\widehat{E}/\mathbb{Q}$ has a prime $q \neq 3$ of split multiplicative reduction and, hence, the curve $E/\mathbb{Q}$ has split multiplicative reduction modulo $q$ (see the proof of \cite[Theorem 5.1]{dd}). This proves that $E/\mathbb{Q}$ has at least two places of split multiplicative reduction. If $E/\mathbb{Q}$ does not have split multiplicative reduction modulo $3$, then line $13$ of \cite[Theorem 6.1]{dd} implies that $\textrm{ord}_3(\frac{c_{3}(\widehat{E})}{c_{3}(E)})=0$ (recall that because $b=1$ the curve $E/\mathbb{Q}$ cannot have reduction of type IV$^*$ and Proposition \ref{prop:3torsionreduction} implies that $E/\mathbb{Q}$ cannot have reduction of type II$^*$). If $p$ is a prime of split multiplicative reduction, then $\textrm{ord}_3(\frac{c_p(\widehat{E})}{c_p(E)})=1$. Moreover, if $r$ is a prime of nonsplit multiplicative reduction, then $\textrm{ord}_3(\frac{c_r(\widehat{E})}{c_r(E)})=0$. Since $E/\mathbb{Q}$ is semi-stable away from $3$, the above prove that $\text{ord}_3(\frac{\prod_{p}c_p(\widehat{E})}{\prod_{p}c_p(E)}) \geq 2.$ This concludes our proof.
\end{proof}

\begin{example}\label{examplesection3}
Let $E/\mathbb{Q}$ be the elliptic curve with Cremona label 27a3 (LMFDB label 27.a4), Cremona label 27a4 (LMFDB label 27.a2), or Cremona label 54a3 (LMFDB label 54.a2). Then $E(\mathbb{Q}) \cong \mathbb{Z}/3\mathbb{Z}$, $|\Sha(E/\mathbb{Q})|=1$, and $c_{\infty}(E)\cdot c(E)=1$. Therefore, $|E(\mathbb{Q})_{\rm{tors}}|$ does not divide $|\Sha(E/\mathbb{Q})|\cdot c_{\infty}(E) \cdot c(E)$. Note however, that in all cases $E/\mathbb{Q}$ is not optimal and has Manin constant equal to $3$.
\end{example}

\section{Semi-stable elliptic curves}

In this section we prove Theorem \ref{maintheoremsemistable}. We first recall the definition of an optimal elliptic curve. Let $E/\mathbb{Q}$ be an elliptic curve of conductor $N$. Work of Breuil, Conrad, Diamond, Taylor, and Wiles (see \cite{bcdt}, \cite{taylorwiles}, and \cite{wiles}) implies that there exists a modular parametrization $\phi : X_0(N) \rightarrow E$ defined over $\mathbb{Q}$, where $X_0(N)/\mathbb{Q}$ is the modular curve associated to the congruence subgroup $\Gamma_0(N)$ (see \cite[Section 1.5]{diamondshurman} for more information on the curve $X_0(N)$). The elliptic curve $E/\mathbb{Q}$ is called optimal if it satisfies the following property: if $E'/\mathbb{Q}$ is an elliptic curve contained in the isogeny class of $E/\mathbb{Q}$ and $\phi': X_0(N) \rightarrow E'$ is a modular parametrization defined over $\mathbb{Q}$, then there exists a $\mathbb{Q}$-isogeny $\beta : E \rightarrow E' $ such that $\phi'=\beta \circ \phi$. Equivalently, $E/\mathbb{Q}$ is optimal if ker$(\psi)$ is connected, where  $\psi : J_0(N) \longrightarrow E $ is the map that induces $\phi$ and $J_0(N)/\mathbb{Q}$ is the Jacobian of the curve $X_0(N)/\mathbb{Q}$.

\begin{proposition}\label{optimalsemistablenotexactly1prime}
Let $E/\mathbb{Q}$ be an optimal semi-stable elliptic curve with analytic rank $0$. If $E/\mathbb{Q}$ does not have exactly one prime of split multiplicative reduction, then $|E(\mathbb{Q})_{\textrm{tors}}|$ divides $ c_{\infty}(E) \cdot|\Sha(E/\mathbb{Q})|\cdot c(E)$.
\end{proposition}
\begin{proof}
 If $|E(\mathbb{Q})_{\textrm{tors}}| \not\cong \mathbb{Z}/3\mathbb{Z}$, then $|E(\mathbb{Q})_{\textrm{tors}}|$ divides $c_{\infty}(E) \cdot c(E)$ with only finitely many exceptions by Parts $(i)$ and $(ii)$ of Theorem \ref{maintheorem}. However, since in each isogeny class of the Cremona database the curve with $1$ at the end of its label is the optimal one (with one exception; the curve with label 990H3 is optimal), we find that none of the exceptions is optimal. Therefore, we can assume from now on that $|E(\mathbb{Q})_{\textrm{tors}}| \cong \mathbb{Z}/3\mathbb{Z}$. Moreover, if $E/\mathbb{Q}$ has more than one prime of split multiplicative reduction, then Theorem \ref{exceptionlist} implies that $|E(\mathbb{Q})_{\textrm{tors}}|$ divides $ |\Sha(E/\mathbb{Q})|\cdot c(E)$. Hence, we can also assume that $E/\mathbb{Q}$ has no prime of split multiplicative reduction. We now show that this case cannot occur. Indeed, if $E/\mathbb{Q}$ has no primes of split multiplicative reduction, then $w(E)=w_{\infty}(E)=-1$ by \ref{rootnumbers}, but this is a contradiction because $w(E)=1=(-1)^{rk(E/\mathbb{Q})}=1$ again by \ref{rootnumbers} because we assume that $E/\mathbb{Q}$ has analytic rank $0$.
\end{proof}

\begin{lemma}\label{lemma3torsionsplit}
Let $p=2$ or $3$. Let $E/\mathbb{Q}$ be an elliptic curve which has split multiplicative reduction modulo $p$ and such that $3$ divides $|E(\mathbb{Q})_{\textrm{tors}}|$. Then $3$ divides $c_p(E)$.
\end{lemma}
\begin{proof}
Let $P$ be a $\mathbb{Q}$-rational point of order $3$. There is a short exact sequence $$0 \longrightarrow E_1(\mathbb{Q}_p) \longrightarrow E_0(\mathbb{Q}_p) \longrightarrow \widetilde{E}_{\rm{ns}}(\mathbb{F}_p)\longrightarrow 0,$$
 where $E_0(\mathbb{Q}_p)$ consists of all points with nonsingular reduction and $E_1(\mathbb{Q}_p)$ is the kernel of the reduction map. We now claim that $P \in E(\mathbb{Q}_p)/ E_0(\mathbb{Q}_p)$, so $3 \mid c_p(E)$ because $c_p(E)= |E(\mathbb{Q}_p)/ E_0(\mathbb{Q}_p)|$. Assume that $P \in E_0(\mathbb{Q}_p)$ and we will find a contradiction. It follows from \cite[Theorem 5.9.4]{Krummthesis} that $E_1(\mathbb{Q}_p)$ cannot have points of order $3$ so $P \not\in E_1(\mathbb{Q}_p)$. This implies that $P$ reduces to a point $\widetilde{P}$ of order $3$ in $\widetilde{E}_{\rm{ns}}(\mathbb{F}_p)$. This is a contradiction because $E/\mathbb{Q}$ has split multiplicative reduction modulo $p$ and, hence, $\widetilde{E}_{\rm{ns}}(\mathbb{F}_p) \cong (\mathbb{F}_p^{*},\cdot)$, which does not contain points of order $3$ because $p=2$ or $3$. This proves our lemma.
\end{proof}

\begin{proposition}\label{prop1primeoptimalsemistable}
Let $E/\mathbb{Q}$ be a semi-stable elliptic curve with exactly one prime $p$ of split multiplicative reduction. Assume also that $E/\mathbb{Q}$ is optimal and that $E(\mathbb{Q})_{\textrm{tors}} \cong \mathbb{Z}/3\mathbb{Z}$. Then $|E(\mathbb{Q})_{\textrm{tors}}|$ divides $c_p(E)$.
\end{proposition}
\begin{proof}

If $p=2$ or $3$, then Lemma \ref{lemma3torsionsplit} implies that $3$ divides $c_p(E)$. Therefore, we can assume from now on that $p > 3$. Since $E/\mathbb{Q}$ has a point of order $3$, proceeding exactly as in the second paragraph of the proof of Theorem \ref{exceptionlist}, we find that $E/\mathbb{Q}$ can be given by a Weierstrass equation of the form 
\begin{align*}
    y^2+axy+by=x^3,
\end{align*}
where $a,b \in \mathbb{Z}$, $b>0$, and for every prime $q$ either $q \nmid a$ or $q^3 \nmid b$. If $p$ is a prime with $p \mid b$, then, since $E/\mathbb{Q}$ is semi-stable, it follows from Proposition \ref{prop:3torsionreduction} that $3 \mid c_p(E)$ and our proposition is proved. Hence, we can assume from now on that $b=1$. Moreover, it follows from work of Miyawaki \cite{miyawaki} that the only semi-stable elliptic curves $E/\mathbb{Q}$ with exactly $1$ prime of bad reduction and a $\mathbb{Q}$-rational point of order $3$ are the curves with Cremona labels 19a1, 19a3, 37b1, and 37b3. Out of those, the curves with Cremona labels 19a1 and 37b1 are optimal and they both satisfy $c(E)=3$. Therefore, we can assume that $E/\mathbb{Q}$ has at least $2$ primes of bad reduction and that $p > 3$.

The rest of our proof is inspired by an argument that appeared in \cite[Page 225]{by3isogeny}. Let $N$ be the conductor of $E/\mathbb{Q}$. Note that $N$ is square free because $E/\mathbb{Q}$ is semi-stable. Since $N$ is square free, using \cite[Page 103]{diamondshurman} we find that a system of representatives for the cusps of $X_0(N)$ can be obtained by considering for each positive divisor $d$ of $N$ all fractions of the form $\frac{a}{d}$, with $(a,d)=1$ and $a$ taken modulo $(d, \frac{N}{d})$ (see also \cite[Section 2]{yoorationalcuspidalx0(N)}). For each positive divisor $d$ of $N$, we define a divisor $(P_d)$ of $X_0(N)$ as the sum of all the cusps of the form $\frac{a}{d}$, with $a$ and $d$ as above.

Write $N=pq_1\cdot \cdot \cdot q_s$, where $p$ and $q_i$ for $i=1,...,s$ are distinct primes. Since $E/\mathbb{Q}$ has split multiplicative reduction modulo $p$ and $p \neq 3$, we obtain, using \cite[Lemma 3.1]{by3isogeny}, that $p \equiv 1 \; ( \text{mod } 3)$. Moreover, since $E/\mathbb{Q}$ has non-split multiplicative reduction modulo $q_i$ for all $i$ we see, again using \cite[Lemma 3.1]{by3isogeny}, that either $q_i \equiv 2 \; ( \text{mod } 3)$ for all $i=1,...,s$ or $q_{i_0}=3$ for some $i_0$ and $q_i \equiv 2 \; ( \text{mod } 3)$ for $i=1,...,s$ with $i \neq i_0$. Note that $3 \nmid q_i-1$ in both of the above cases.

Let $\phi : X_0(N) \rightarrow E$ be the modular parametrization and let $\psi: J_0(N) \rightarrow E $ be the induced morphism. We denote by $\widehat{\psi}: E(\mathbb{Q})_{\rm{tors}} \rightarrow J_0(N)(\mathbb{Q})_{\rm{tors}}$ the map induced by the dual of $\psi$. Let $\Phi_N(p)$ be the group of components of the N\'eron model of $J_0(N)_{\mathbb{Q}_p}/{\mathbb{Q}_p}$ and let $\pi_{N,p}: J_0(N)(\mathbb{Q})_{\rm{tors}} \rightarrow \Phi_N(p)$ be the reduction map. If $r: E(\mathbb{Q})_{\rm{tors}} \rightarrow E(\mathbb{Q}_p)/E_0(\mathbb{Q}_p)$ is the reduction map of $E/\mathbb{Q}$ and $\widehat{\psi}':E(\mathbb{Q}_p)/E_0(\mathbb{Q}_p) \rightarrow \Phi_N(p)$ is the map induced by by the dual of $\psi$, then $\widehat{\psi}' \circ r = \pi_{N,p} \circ \widehat{\psi}$ (see \cite[Page 225]{by3isogeny}). 

Let $N'=q_1\cdot \cdot \cdot q_s$. By \cite[Theorem 2.3]{by3isogeny} (see also \cite{byeonyhee2011} and \cite{dummigan2005}) we find that $E/\mathbb{Q}$ has a $\mathbb{Q}$-rational point $P$ of order $3$ such that $$\widehat{\psi}(P)=\frac{2(p-1)}{3h}\displaystyle\prod_{i=1}^s(q_i^2-1)[(P_{N'})-(P_N)],$$ where $h=(u,24)$ and $u=(p-1)\displaystyle\prod_{i=1}^s[(q_i^2-1)(p-1)]$. We note that $3$ divides $h$ because $p \equiv 1 \; ( \text{mod } 3)$. Since $\pi_{N,p}((P_{N'})-(P_N))=(0)-(\infty) $ (see \cite[Section 2.4]{by3isogeny}), we find that $$\pi_{N,p}( \widehat{\psi}(P))=\frac{2(p-1)}{3h}\displaystyle\prod_{i=1}^s(q_i^2-1)[(0)-(\infty)].$$ The order of $(0)-(\infty)$ in $\Phi_N(p)$ is $\frac{p-1}{k}\displaystyle\prod_{i=1}^s (q_i+1)$, for some $k$ which is equal to $2,4,6,$ or $12$ (see \cite[Theorem 2.4]{by3isogeny}). Therefore, since $3 \mid h$ and $3 \nmid q_i-1$ for all $i$, we find that $\pi_{N,p}( \widehat{\psi}(P))$ is non-trivial in $\Phi_N(p)$. This implies that we must have that $P \not\in E_0(\mathbb{Q}_p)$ because $\widehat{\psi}' \circ r = \pi_{N,p} \circ \widehat{\psi}$. Finally, since $P$ has order $3$ we see that $3 \mid c_p(E)$, which proves our proposition.
\end{proof}

\begin{proof}[Proof of Theorem \ref{maintheoremsemistable}]
Let $E/\mathbb{Q}$ be an optimal semi-stable elliptic curve with analytic rank $0$. Since $E/\mathbb{Q}$ is semi-stable we must have that $j_E \neq 0, 1728$. Indeed, every elliptic curve with $j$-invariant $j_E=0$ or $1728$ has everywhere potentially good reduction by \cite[Proposition VII.5.5]{aec} and if $E/\mathbb{Q}$ is, in addition, semi-stable then it has everywhere good reduction. However, there is no elliptic curve over $\mathbb{Q}$ that has everywhere good reduction. This proves that $j_E \neq 0, 1728$. If $E/\mathbb{Q}$ does not have exactly one prime of split multiplicative reduction, then $|E(\mathbb{Q})_{\textrm{tors}}|$ divides $c_{\infty}(E) \cdot |\Sha(E/\mathbb{Q})|\cdot c(E)$ by Proposition \ref{optimalsemistablenotexactly1prime}. Therefore, we can assume from now on that $E/\mathbb{Q}$ has exactly one prime of split multiplicative reduction. If $|E(\mathbb{Q})_{\textrm{tors}}| \not\cong \mathbb{Z}/3\mathbb{Z}$, then $|E(\mathbb{Q})_{\textrm{tors}}|$ divides $ c_{\infty}(E) \cdot c(E)$ with only finitely many exceptions by Parts $(i)$ and $(ii)$ of Theorem \ref{maintheorem}. However, none of the exceptions is optimal. Finally, if $|E(\mathbb{Q})_{\textrm{tors}}| \cong \mathbb{Z}/3\mathbb{Z}$, then $|E(\mathbb{Q})_{\textrm{tors}}|$ divides $c_p(E)$ by Proposition \ref{prop1primeoptimalsemistable}. This proves our theorem.
\end{proof}

\bibliographystyle{plain}
\bibliography{bibliography.bib}

\begin{thebibliography}{10}

\bibitem{agasheribetstein}
A.~Agashe, K.~Ribet, and W.~A. Stein.
\newblock The {M}anin constant.
\newblock {\em Pure Appl. Math. Q.}, 2(2, Special Issue: In honor of John H.
  Coates. Part 2):617--636, 2006.

\bibitem{agashestein}
A.~Agashe and W.~A. Stein.
\newblock Visible evidence for the {B}irch and {S}winnerton-{D}yer conjecture
  for modular abelian varieties of analytic rank zero.
\newblock {\em Math. Comp.}, 74(249):455--484, 2005.
\newblock With an appendix by J. Cremona and B. Mazur.

\bibitem{bcdt}
C.~Breuil, B.~Conrad, F.~Diamond, and R.~Taylor.
\newblock On the modularity of elliptic curves over {$\bold Q$}: wild 3-adic
  exercises.
\newblock {\em J. Amer. Math. Soc.}, 14(4):843--939, 2001.

\bibitem{bky2}
D.~Byeon, T.~Kim, and D.~Yhee.
\newblock A conjecture of {G}ross and {Z}agier: case {$E(\Bbb Q)_{\rm tor}\cong
  \Bbb Z/{\bf 2}\Bbb Z\oplus \Bbb Z/{\bf 2}\Bbb Z, \Bbb Z/{\bf 2}\Bbb
  Z\oplus\Bbb Z/{\bf 4}\Bbb Z$} or {$\Bbb Z/{\bf 2}\Bbb Z\oplus\Bbb Z/{\bf
  6}\Bbb Z$}.
\newblock {\em Int. J. Number Theory}, 16(7):1567--1572, 2020.

\bibitem{byeonyhee2011}
D.~Byeon and D.~Yhee.
\newblock Rational torsion on optimal curves and rank-one quadratic twists.
\newblock {\em J. Number Theory}, 131(3):552--560, 2011.

\bibitem{by3isogeny}
D.~Byeon and D.~Yhee.
\newblock Optimal curves differing by a 3-isogeny.
\newblock {\em Acta Arith.}, 158(3):219--227, 2013.

\bibitem{cesnaviciusmaninsemistable}
K.~{\v C}esnavičius.
\newblock The {M}anin constant in the semistable case.
\newblock {\em Compos. Math.}, 154(9):1889--1920, 2018.

\bibitem{connel}
I.~Connell.
\newblock Calculating root numbers of elliptic curves over {${\bf Q}$}.
\newblock {\em Manuscripta Math.}, 82(1):93--104, 1994.

\bibitem{cremonabook}
J.~E. Cremona.
\newblock {\em Algorithms for modular elliptic curves}.
\newblock Cambridge University Press, Cambridge, second edition, 1997.

\bibitem{darmonmodularellipticcurves}
H.~Darmon.
\newblock {\em Rational points on modular elliptic curves}, volume 101 of {\em
  CBMS Regional Conference Series in Mathematics}.
\newblock Published for the Conference Board of the Mathematical Sciences,
  Washington, DC; by the American Mathematical Society, Providence, RI, 2004.

\bibitem{diamondshurman}
F.~Diamond and J.~Shurman.
\newblock {\em A first course in modular forms}, volume 228 of {\em Graduate
  Texts in Mathematics}.
\newblock Springer-Verlag, New York, 2005.

\bibitem{dd1}
T.~Dokchitser and V.~Dokchitser.
\newblock On the {B}irch--{S}winnerton-{D}yer quotients modulo squares.
\newblock {\em Ann. of Math. (2)}, 172(1):567--596, 2010.

\bibitem{dd}
T.~Dokchitser and V.~Dokchitser.
\newblock Local invariants of isogenous elliptic curves.
\newblock {\em Trans. Amer. Math. Soc.}, 367(6):4339--4358, 2015.

\bibitem{dummigan2005}
N.~Dummigan.
\newblock Rational torsion on optimal curves.
\newblock {\em Int. J. Number Theory}, 1(4):513--531, 2005.

\bibitem{grosszagierpaper}
B.~H. Gross and D.~B. Zagier.
\newblock Heegner points and derivatives of {$L$}-series.
\newblock {\em Invent. math.}, 84(2):225--320, 1986.

\bibitem{had}
T.~Hadano.
\newblock Elliptic curves with a torsion point.
\newblock {\em Nagoya Math. J.}, 66:99--108, 1977.

\bibitem{halberstadt}
E.~Halberstadt.
\newblock Signes locaux des courbes elliptiques en 2 et 3.
\newblock {\em C. R. Acad. Sci. Paris S\'{e}r. I Math.}, 326(9):1047--1052,
  1998.

\bibitem{hindrysilverman}
M.~Hindry and J.~H. Silverman.
\newblock {\em Diophantine geometry}, volume 201 of {\em Graduate Texts in
  Mathematics}.
\newblock Springer-Verlag, New York, 2000.
\newblock An introduction.

\bibitem{hus}
D.~Husem\"{o}ller.
\newblock {\em Elliptic curves}, volume 111 of {\em Graduate Texts in
  Mathematics}.
\newblock Springer-Verlag, New York, second edition, 2004.
\newblock With appendices by Otto Forster, Ruth Lawrence and Stefan Theisen.

\bibitem{kolyvagineulersystems}
V.~A. Kolyvagin.
\newblock Euler systems.
\newblock In {\em The {G}rothendieck {F}estschrift, {V}ol. {II}}, volume~87 of
  {\em Progr. Math.}, pages 435--483. Birkh\"{a}user Boston, Boston, MA, 1990.

\bibitem{kolyvaginlogachev89}
V.~A. Kolyvagin and D.~Yu. Logach\"{e}v.
\newblock Finiteness of the {S}hafarevich--{T}ate group and the group of
  rational points for some modular abelian varieties.
\newblock {\em Algebra i Analiz}, 1(5):171--196, 1989.

\bibitem{kolyvaginlogachev92}
V.~A. Kolyvagin and D.~Yu. Logach\"{e}v.
\newblock Finiteness of \foreignlanguage{russian}{Ш} over totally real fields.
\newblock {\em Izv. Akad. Nauk SSSR Ser. Mat.}, 55(4):851--876, 1991.

\bibitem{koz}
R.~Kozuma.
\newblock A note on elliptic curves with a rational 3-torsion point.
\newblock {\em Rocky Mountain J. Math.}, 40(4):1227--1255, 2010.

\bibitem{Krummthesis}
D.~Krumm.
\newblock {\em Quadratic Points on Modular Curves}.
\newblock PhD thesis, University of Georgia, 2013.

\bibitem{lmfdb}
The {LMFDB Collaboration}.
\newblock The {L}-functions and modular forms database.
\newblock \url{http://www.lmfdb.org}, 2019.
\newblock [Online; accessed 24 June 2021].

\bibitem{lor}
D.~J. Lorenzini.
\newblock Torsion and {T}amagawa numbers.
\newblock {\em Ann. Inst. Fourier (Grenoble)}, 61(5):1995--2037 (2012), 2011.

\bibitem{maz}
B.~Mazur.
\newblock Modular curves and the {E}isenstein ideal.
\newblock {\em Inst. Hautes \'{E}tudes Sci. Publ. Math.}, (47):33--186 (1978),
  1977.
\newblock With an appendix by Mazur and M. Rapoport.

\bibitem{mentzelosagasheconjecture}
M.~Melistas.
\newblock On a conjecture of {A}gashe.
\newblock {\em Trans. Amer. Math. Soc.}, 374:7143--7160, (2021).
\newblock DOI: https://doi.org/10.1090/tran/8419.

\bibitem{mestreoesterle1989}
J.-F. Mestre and J.~Oesterl\'{e}.
\newblock Courbes de {W}eil semi-stables de discriminant une puissance
  {$m$}-i\`eme.
\newblock {\em J. Reine Angew. Math.}, 400:173--184, 1989.

\bibitem{miyawaki}
I.~Miyawaki.
\newblock Elliptic curves of prime power conductor with {${\bf Q}$}-rational
  points of finite order.
\newblock {\em Osaka Math. J.}, 10:309--323, 1973.

\bibitem{pap}
I.~Papadopoulos.
\newblock Sur la classification de {N}\'{e}ron des courbes elliptiques en
  caract\'{e}ristique r\'{e}siduelle {$2$} et {$3$}.
\newblock {\em J. Number Theory}, 44(2):119--152, 1993.

\bibitem{rab}
F.~P. Rabarison.
\newblock Structure de torsion des courbes elliptiques sur les corps
  quadratiques.
\newblock {\em Acta Arith.}, 144(1):17--52, 2010.

\bibitem{rohrlich}
D.~E. Rohrlich.
\newblock Variation of the root number in families of elliptic curves.
\newblock {\em Compositio Math.}, 87(2):119--151, 1993.

\bibitem{silverman2}
J.~H. Silverman.
\newblock {\em Advanced topics in the arithmetic of elliptic curves}, volume
  151 of {\em Graduate Texts in Mathematics}.
\newblock Springer-Verlag, New York, 1994.

\bibitem{aec}
J.~H. Silverman.
\newblock {\em The arithmetic of elliptic curves}, volume 106 of {\em Graduate
  Texts in Mathematics}.
\newblock Springer, Dordrecht, second edition, 2009.

\bibitem{tatealgorithm}
J.~Tate.
\newblock Algorithm for determining the type of a singular fiber in an elliptic
  pencil.
\newblock In {\em Modular functions of one variable, {IV} ({P}roc. {I}nternat.
  {S}ummer {S}chool, {U}niv. {A}ntwerp, {A}ntwerp, 1972)}, pages 33--52.
  Lecture Notes in Math., Vol. 476, 1975.

\bibitem{taylorwiles}
R.~Taylor and A.~Wiles.
\newblock Ring-theoretic properties of certain {H}ecke algebras.
\newblock {\em Ann. of Math. (2)}, 141(3):553--572, 1995.

\bibitem{sagemath}
{The Sage Developers}.
\newblock {\em {S}ageMath, the {S}age {M}athematics {S}oftware {S}ystem
  ({V}ersion 9.0)}, 2020.
\newblock {\tt https://www.sagemath.org}.

\bibitem{wiles}
A.~Wiles.
\newblock Modular elliptic curves and {F}ermat's last theorem.
\newblock {\em Ann. of Math. (2)}, 141(3):443--551, 1995.

\bibitem{yoorationalcuspidalx0(N)}
Hwajong Yoo.
\newblock The rational cuspidal divisor class group of ${X}_0({N})$.
\newblock {\em Journal of Number Theory}, 2022.

\end{thebibliography}

\end{document}